\documentclass{article}
\usepackage[top=2.5cm,bottom=2.5cm,left=2.5cm,right=2.5cm]{geometry}
\usepackage[english]{babel}

\usepackage{graphicx}
\usepackage{xcolor}
\usepackage{hyperref}
\hypersetup{
     colorlinks = true,
     citecolor  = blue,
     linkcolor = blue,
     urlcolor = blue
}
\usepackage{xcolor}
\usepackage[fleqn]{amsmath}
\usepackage{amssymb,amsthm}
\newtheorem{theorem}{Theorem}[section]

\usepackage{cite}
\usepackage{subfig}
\usepackage{float}
\usepackage{authblk}
\usepackage[ruled,linesnumbered,noresetcount,vlined]{algorithm2e}

\SetKwFor{Repeat}{repeat}{:}{endw}
\SetKwFor{Forall}{forall}{}{endw}

\usepackage{amsmath}
\DeclareMathOperator*{\argmin}{argmin}

\usepackage{tikz}
\usetikzlibrary{arrows, arrows.meta,automata,chains,fit,positioning,shapes}

\title{Column Generation \\
for Real-Time Ride-Sharing Operations}
\author[1]{\mbox{Connor Riley}}
\author[2]{\mbox{Antoine Legrain}}
\author[1]{\mbox{Pascal Van Hentenryck}}
\affil[1]{\mbox{Georgia Institute of Technology}, Atlanta GA, USA}
\affil[2]{Polytechnique Montréal, Montreal QC H3T 1J4, Canada}
\date{}

\usepackage{enumitem}
\newcommand{\name}{{\sc RTDARS}}

\begin{document}
\graphicspath{{images/}}

\maketitle

\begin{abstract}
    This paper considers real-time dispatching for large-scale
    ride-sharing services over a rolling horizon. It presents \name{}
    which relies on a column-generation algorithm to minimize wait
    times while guaranteeing short travel times and service for each
    customer.  Experiments using historic taxi trips in New York City
    for instances with up to 30,000 requests per hour indicate that
    the algorithm scales well and provides a principled and effective
    way to support large-scale ride-sharing services in dense cities.

\end{abstract}
\section{Introduction}

In the past decade, commercial ride-hailing services such as Didi,
Uber, and Lyft have decreased reliance on personal vehicles and
provided new mobility options for various population segments. More recently, 
ride-sharing has been introduced as an option for customers
using these services. Ride-sharing has the potential for significant
positive impact since it can reduce the number of cars on the roads
and thus congestion, decrease greenhouse emissions, and make mobility
accessible to new population segments by decreasing trip prices.
However, the algorithms used by commercial ride-sharing services
rarely use state-of-the-art techniques, which reduces the potential
positive impact. Recent research by Alonso-Mora et
al. \cite{Alonso-Mora462} has shown the benefits of more sophisticated
algorithms. Their algorithm uses shareability
graphs and cliques to generate all possible routes and a MIP model to
select the routes. They impose significant constraints on waiting
times (e.g., 420 seconds), which reduces the potential riders to
consider for each route at the cost of rejecting customers.

This paper considers large-scale ride-sharing services where {\em
  customers are always guaranteed a ride}, in contrast to prior
work. The Real-Time Dial-A-Ride System (\name{}) divides the days into
short time periods called epochs, batches requests in a given epoch,
and then schedules customers to minimize average waiting
times. \name{} makes a number of modeling and solving contributions. At
the modeling level, \name{} has the following innovations:

\begin{enumerate}[wide, labelindent=5pt]

\item \name{} follows a Lagrangian approach, relaxing the constraint
  that all customers must be served in the static optimization problem
  of each epoch. Instead, \name{} associates a penalty with each rider,
  representing the cost of not serving the customer.

\item To balance the minimization of average waiting times and ensure
  that the waiting time of every customer is reasonable, \name{}
  increases the penalty of an unserved customer in the next epoch,
  making it increasingly harder not to serve the waiting rider.

\item \name{} exploits a key property of the resulting formulation to
  reduce the search space explored for each epoch.

\item To favor ride-sharing, \name{} uses the concept of virtual stops
  used in the RITMO project \cite{RITMO} and being adopted by
  ride-sourcing services.
\end{enumerate}
\name{} solves the static optimization problem for each epoch with a
column-generation algorithm based on the three-index MIP formulation
\cite{Cordeau2007}. The main innovation here is the pricing problem
which is organized as a series of waves, first considering all the
insertions of a single customer, before incrementally adding more
customers.

\name{} was evaluated on historic taxi trips from the New York City
Taxi and Limousine Commission \cite{nycdata}, which contains
large-scale instances with more than 30,000 requests an hour. The
results show that \name{} can provide service guarantees while
improving the state-of-the-art results. For instance, for a fleet of
2,000 vehicles of capacity 4, \name{} obtains an average wait of 2.2
minutes and an average deviation from the shortest path of 0.62
minutes. The results also show that large-occupancy vehicles (e.g.,
8-passenger vehicles) provide additional benefits in terms of waiting
times with negligible increases in in-vehicle time. \name{} is also
shown to generate a small fraction of the potential columns,
explaining its efficiency. The Lagrangian modeling also helps in
reducing computation times significantly.

The rest of this paper is organized as follows. Section
\ref{section-related} presents the related work in more
detail. Section \ref{section-online} describes the real-time
setting. Section \ref{section-static} specifies the static problem and
gives the MIP formulation. Section \ref{section-cg} describes the
column generation. Section \ref{section-rt} specifies the real-time
operations. Section \ref{section-results} presents the experimental
results and Section \ref{section-conclusion} concludes the paper.

\section{Related Work}
\label{section-related}

Dial-a-ride problems have been a popular topic in operations research
for a long time. Cordeau and Laporte \cite{Cordeau2007} provided a
comprehensive review of many of the popular formulations and the
starting point of \name{}'s column generation is their three-index
formulation. Constraint programming and large neighborhood search were
also proposed for dial-a-ride problems (e.g., \cite{Jain2011}
\cite{Berbeglia2012}).  Progress in communication technologies and the
emergence of ride-sourcing and ride-sharing services have stimulated
further research in this area. Rolling horizons are often used to
batch requests and were used in taxi pooling previously \cite{stars,
  scalable-taxi}. In addition, stochastic scenarios along with waiting
and reallocation strategies have been previously explored in
\cite{Bent2007,scenariopvh}. Bertsimas, Jaillet, and Martin
\cite{Bertsimas2018OnlineVR} explored the taxi routing problem
(without ride-sharing) and introduced a ``backbone'' algorithm which
increases the sparsity of the problem by computing a set of candidate
paths that are likely to be optimal. Alonso-Mora et al. proposed an
anytime algorithm which uses cliques to generate vehicle paths
combined with a vehicle rebalancing step to move vehicles towards
demand \cite{Alonso-Mora462}. Their ``results show that 2,000 vehicles
(15\% of the taxi fleet) of capacity 10 or 3,000 of capacity 4 can
serve 98\% of the demand within a mean waiting time of 2.8 min and
mean trip delay of 3.5 min.''  \cite{Alonso-Mora462}.  Both
\cite{Alonso-Mora462} and \cite{Bertsimas2018OnlineVR} use hard time
windows to reject riders when they cannot serve them quickly enough
(e.g., 420 seconds in the aforementioned results). This decision
significantly reduces the search space as only close riders can be
served by a vehicle. In contrast, \name{} provides service guarantees
for all riders, while still reducing the search space through a
Lagrangian reformulation. The results show that \name{} is capable of
providing these guarantees while improving prior results in terms of
average waiting times. Indeed, for 2,000 vehicles of capacity 4,
\name{} provides an average waiting time of 2.2 minutes with a
standard deviation of 1.24 and a mean trip deviation of 0.62 minutes
(standard deviation 1.13). For 3,000 vehicles of capacity 4, the
average waiting time is further reduced to 1.81 minutes with a
standard deviation of 1.03 and an average trip deviation of 0.23
minutes.

\section{Overview of the Approach}
\label{section-online}

\name{} divides time into epochs, e.g., time periods of 30
seconds. During an epoch, \name{} performs two tasks: It batches
incoming requests and it solves the epoch optimization problem for all
unserved customers from prior epochs. The epoch optimization takes, as
inputs, these unserved customers and their penalties, as well as the
{\em first} stop of each vehicle after the start of the epoch: Vehicle
schedules prior to this stop are committed since, for safety reasons,
\name{} does not allow a vehicle to be re-routed once it has departed
for its next customer. These first stops are called {\em departing
  stops} in this paper. All customers served before and up to the
departing stops of the vehicles are considered served. All others,
even if they were assigned a vehicle in the prior epoch optimization,
are considered unserved.

Once the epoch is completed, a new schedule and a new set of requests
are available. The schedule commits the vehicle routes for the
entire next epoch and determines their next departing stops. The
customer penalties are also updated to make it increasingly harder not
to serve them. \name{} then moves to the next epoch.

\section{The Static Problem}
\label{section-static}

This section defines and presents the static (generalized) dial-a-ride
problem solved for each epoch. its objective is to schedule a
set of requests on a given set of vehicles while ensuring that no
customer deviates too much from their shortest trip time.

The inputs consist primarily of the vehicle and request data.  The set
of vehicles is denoted by $V$ and each vehicle $v \in V$ is associated
with a tuple $(u^v_0, w^v_0, I_v, T_v^B, T_v^E, Q_v)$, where $u^v_0$
is the time the vehicle arrives at its {\em departing stop} for the
epoch, $w^v_0$ is the number of passengers currently in the vehicle,
$I_v$ is the set of dropoff requests for on-board passengers, $T^B_v$
is the vehicle start time, $T^E_v$ is the vehicle end time, and $Q_v$
is the capacity of the vehicle. In other words, a vehicle $v$ can only
insert new requests after time $u^v_0$ and it must serve the dropoffs
in $I_v$. The request data is given in terms of a complete graph
$\mathcal{G} = (\mathcal{N}, \mathcal{A})$, which contains the nodes
for each possible pickup and delivery. There are five types of nodes:
the pickup nodes $P = \{1, \dots n\}$, their associated dropoff nodes
$D = \{n+1, \dots 2n\}$, the dropoff nodes $I = \cup_{v \in V} I_v$ of
the passengers inside the vehicles, the source~$0$, and the sink~$s$
(the last node in terms of indices). Each node $i$ is associated with
a number of people~$q_i$ to pick up ($q_i > 0$) or drop off ($q_i <
0$) and the time $\Delta_i \geq 0$ it takes to perform them. If $i \in
P$, then the corresponding delivery node is $n + i$ and $q_i = -
q_{n+i}$.  Also, $q_i$ and $\Delta_i$ are zero for the source and the
sink. Each node $i \in P$ is associated with a request, which is a
tuple of the form $(e_i,o_i,d_i,q_i)$ where $e_i$ is the earliest
possible pickup time, $o_i$ is the pickup location, $d_i$ is the
dropoff location, and $q_i$ is the number of passengers. Every
request~$i$ in $I$ is associated with the time~$u^P_i$ on which the
request was picked up.  Every request~$i \in P \cup I$ is associated
with the shortest time~$t_i$ from the request origin to its
destination.  Finally, the input contains a matrix $(t_{i,j})_{(i,j)
  \in \mathcal{A}}$ of travel times from any node $i$ to any node $j$
satisfying the triangle inequality, the constants~$\alpha$ and $\beta$
which constrain the deviation from the shortest path, and the
penalty~$p_i$ of not serving the request~$i \in P$.

\begin{figure}[!t]
\begin{subequations} \label{model:static}
\begin{align}
\min \;\;\;\;\;	&   \quad \sum_{i \in P} \sum_{v \in V} (u^v_i - e_i) +  \sum_{i \in P} p_i z_i \label{model:static_obj} \\
\intertext{subject to} 
	&  \left (\sum_{v \in V} \sum_{j \in \mathcal{N}} x^v_{ij} \right ) + z_i = 1 & \forall i \in P \label{model:static_constr:allserved} \\ 
	& \sum_{j \in \mathcal{N}} x^v_{ij} = \sum_{j \in \mathcal{N}} x^v_{ji} & \forall i \in \mathcal{N}\setminus\{0,s\}, \forall v \in V \label{model:static_constr:flow} \\
	& \sum_{j \in \mathcal{N}} x^v_{0j} = 1 & \forall v \in V \label{model:static_constr:source} \\
	& \sum_{j \in \mathcal{N}} x^v_{j,s} = 1 & \forall v \in V \label{model:static_constr:sink} \\
	& \sum_{j \in \mathcal{N}} x^v_{ij} - \sum_{j \in \mathcal{N}}x^v_{n+i, j} = 0 & \forall i \in P, \forall v \in V \label{model:static_constr:droppick} \\
	& \sum_{i \in \mathcal{N}} x^v_{ij} = 1 & \forall j \in I_v, \forall v \in V \label{model:static_constr:dropoffs} \\
    & u^v_j \geq (u^v_i + \Delta_i + t_{ij})x^v_{ij} & \forall i,j \in \mathcal{N}, \forall v \in V \label{model:static_constr:vstart} \\
    & u^v_0 \geq T^B_v & \forall v \in V \label{model:static_constr:voperation1} \\
    & u^v_s \leq T^E_v & \forall v \in V \label{model:static_constr:voperation2} \\
    & u^v_i \geq e_i & \forall i \in P, v \in V \label{model:static_constr:bounds} \\
    & t_i \leq u^v_{n+i} - (u^v_i + \Delta_i) \leq \max\{ \alpha t_i, \beta + t_i\} & \forall i \in P, \forall v \in V \label{model:static_constr:traveltime} \\
    & t_{i} \leq u^v_{i} - (u^P_i + \Delta_i) \leq \max\{ \alpha t_{i}, \beta + t_{i}\} & \forall i \in I_v, \forall v \in V \label{model:static_constr:traveltime_passengers} \\
    & w^v_j \geq (w^v_i + q_j)x^v_{ij} & \forall i,j \in \mathcal{N}, \forall v \in V \label{model:static_constr:calccapacity} \\
    & 0 \leq w^v_i \leq Q_v & \forall i \in \mathcal{N}, \forall v \in V \label{model:static_constr:capcity} \\
    & x^v_{ij} \in \{0, 1\} & \forall i,j \in \mathcal{N}, \forall v \in V \label{model:static_constr:domain}
\end{align}
\end{subequations}
\caption{The Static Formulation of the Dial-A-Ride Problem.} \label{fig:static}
\end{figure}

A MIP model for the static problem is presented in
Figure~\ref{fig:static}. The MIP variables are as follows: $u^v_i$
represents the time at which vehicle $v$ arrives at node $i$, $w^v_i$
the number of people in vehicle $v$ when $v$ leaves node $i$,
$x^v_{ij}$ denotes whether edge $(i,j)$ is used by vehicle $v$, and
$z_i$ captures whether request $i \in P$ is served. Objective
\eqref{model:static_obj} balances the minimization of wait times for
every pickups with the penalties incurred by unserved riders.  Note
that the wait times for riders in $I$ are not included in the
objective because these riders are already in vehicles: only the
constraints on their deviations must be satisfied.
Constraints~\eqref{model:static_constr:allserved} ensure that only one
vehicle serves each request and that, if the request is not served,
$z_i$ is set to 1 to activate the penalty in the objective.
Constraints~\eqref{model:static_constr:flow} are flow balance
constraints.  Constraints~\eqref{model:static_constr:source} and
\eqref{model:static_constr:sink} are flow constraints for the source
and the sink. Constraints~\eqref{model:static_constr:droppick} ensure
that every request is dropped off by the same vehicle that picks it
up. Constraints~\eqref{model:static_constr:dropoffs} ensure that every
passenger currently in a vehicle is dropped off.
Constraints~\eqref{model:static_constr:vstart} define the arrival
times at the
nodes. Constraints~\eqref{model:static_constr:voperation1} and
\eqref{model:static_constr:voperation2} ensure that the vehicle is
operational during its working hours.
Constraints~\eqref{model:static_constr:bounds} ensure that each rider
is picked up no earlier than its lower bound.
Constraints~\eqref{model:static_constr:traveltime} ensure that the
travel time of each served passenger does not deviate too much from
the shortest path between its origin and destination.  Passengers are
allowed to spend either $\alpha * t_i$ (a percentage of the shortest
path), or $\beta + t_i$ (a constant deviation time from the shortest
path) traveling in the vehicle, whichever is larger.
Constraints~\eqref{model:static_constr:traveltime_passengers} do the
same for passengers already in a
vehicle. Constraints~\eqref{model:static_constr:calccapacity} define
the vehicle capacities. Lastly,
constraints~\eqref{model:static_constr:capcity} ensure that the
vehicle capacities are not exceeded.
Constraints~\eqref{model:static_constr:vstart} and
\eqref{model:static_constr:calccapacity} can be linearized using a Big
$M$ formulation.

The following theorem provides a way to prune the search space
significantly. It shows that, in an optimal solution, a rider cannot
be picked up by a vehicle $v$ if the smallest possible wait time
incurred using $v$ is greater than her penalty.

\begin{theorem}
  A feasible solution where rider $l$ is assigned to vehicle $v$ such
  that $u^{v}_0 + t_{0,l} - e_l > p_l$ is suboptimal. \label{thm:1}
\end{theorem}

\begin{proof}
Suppose that there exists a feasible solution~{\it (I)} that serves a
passenger $l$ such that $u^{v}_0 + t_{0,l} - e_l > p_l$.  Let $r$ be
the route of vehicle $v$ (i.e., a sequence of edges in
$\mathcal{A}$). Removing the pickup and dropoff of rider~$l$ from
route~$r$ produces a new feasible route~$\hat{r}$ since the deviation
time cannot increase by the triangular inequality and the number of
riders in $v$ decreases.  Solution~{\it(II)} is derived from solution~{\it (I)}
by replacing the route~$r$ by route~$\hat{r}$ and fixing $z_l$ to
1. Using $\hat{u}$ and $\hat{z}$ to denote the variables of solution~{\it (II)}, the cost $C_{{\it (II)}}$ of solution~{\it (II)} is:
\begin{subequations}
\begin{align}
  C_{{\it (II)}} &= \sum_{i \in P\setminus\{l\}} \sum_{v \in V} (\hat{u}^v_i - e_i) +  \sum_{i \in P \setminus\{l\}} p_i \hat{z_i} + p_l \label{thm1:sol2}\\
    &< \sum_{i \in P\setminus\{l\}} \sum_{v \in V} (\hat{u}^v_i - e_i) +  \sum_{i \in P \setminus\{l\}} p_i \hat{z}_i + u^{v}_0 + t_{0,l} - e_l \label{thm1:hyp}\\
    &\leq \sum_{i \in P\setminus\{l\}} \sum_{v \in V} (u^v_i - e_i) +  \sum_{i \in P \setminus\{l\}} p_i \hat{z}_i + u^v_l - e_l \label{thm1:triang}\\
    &= \sum_{i \in P} \sum_{v \in V} (u^v_i - e_i) + \sum_{i  \in P} p_i z_i = C_{(I)} \label{thm1:sol1}
\end{align}
\end{subequations}
Equality~\eqref{thm1:sol2} is just the definition of the objective of
solution~{\it (II)}. Inequality~\eqref{thm1:hyp} is induced by the
hypothesis. Inequality~\eqref{thm1:triang} is induced by the
triangular inequality on the travel times. Inequality~\eqref{thm1:sol1}
just factors the equation to get the objective of
solution~{\it (I)}. Solution~{\it (I)} is thus suboptimal.
\end{proof}

\section{The Column-Generation Algorithm} \label{sec:algorithm}
\label{section-cg}

This section presents the column-generation algorithm, starting with
the master problem before presenting the pricing subproblem, and the
specifics of the column-generation process. Upon completion of the
column generation, \name{} solves a final MIP that imposes integrality
constraints on the master problem variables.

\paragraph{The Master Problem}

The restricted master problem, RMP, (presented in Figure~\ref{fig:master}) selects a
route for each vehicle.  In order for a route to be assigned to a
vehicle, the route must contain dropoffs for every current passenger
of that vehicle.  The set of routes is denoted by $R$ and its subset of
routes that can be assigned to vehicle $v$ is denoted $R_v$.  The
variables in the master problem are the following: $y_r \in [0,1]$
is set to 1 if potential route $r$ is selected for use and variable
$z_{i} \in [0,1]$ is set to 1 if request $i$ is not served by any of
the selected routes.  The constants are as follows: $c_r$ is the sum
of the wait time incurred by customers served by route $r$,
$p_{i}$ is the cost of not scheduling request $i$ for this period, and
$a_{i}^{r}=1$ if request $i$ is served by route $r$.  The objective
minimizes the waiting times incurred by all customers on each route
and the penalties for the customers not scheduled during the current
period.  Constraints~\eqref{model:master_constr:allserved} ensure that
$z_{i}$ is set to 1 if request $i$ is not served by any of the
selected routes and constraints~\eqref{model:master_constr:oneroute}
ensure that only one route is selected per vehicle.  The dual
variables associated with each constraint are specified in between parentheses
next to the constraint in the model.

\begin{figure}[!t]
\begin{subequations} \label{model:master}
\begin{align}
\min 	&  \quad \sum_{r \in R}  c_r y_r + \sum_{i \in P} p_{i} z_{i} \label{model:master_obj}\\
\mbox{subject to} \\
	&  \left ( \sum_{r \in R} y_r a_{i}^{r} \right ) + z_{i} = 1 & \forall i \in P && (\pi_i) \label{model:master_constr:allserved} \\ 
	&  \sum_{r \in R_v} y_r = 1 & \forall v \in V && (\sigma_v) \label{model:master_constr:oneroute} \\
	&  z_{i} \in \mathbb{N} & \forall i \in P \label{model:master_constr:domainz}\\
	&  y_r \in \{0,1\} & \forall r \in R \label{model:master_constr:domainy}
\end{align}
\end{subequations}
\caption{The Master Problem Formulation.} \label{fig:master}
\end{figure}

\paragraph{The Pricing Problem}

The routes for each vehicle~$v$ are generated via a pricing problem
depicted in Figure~\ref{fig:pricing}. The pricing
problem~\eqref{model:sub} is defined for a given
vehicle~$v$. Theorem~\ref{thm:1} makes it possible to remove some
passengers from the set~$P$ to obtain the subset~$P_v$ and thus a new
graph~$\mathcal{G}_v = (\mathcal{N}_v, \mathcal{A}_v)$. The 
pricing problem minimizes the reduced cost of the route
being generated.  Constraints \eqref{model:sub_constr:flow} --
\eqref{model:sub_constr:domain} correspond to constraints
\eqref{model:static_constr:flow} -- \eqref{model:static_constr:domain}
in the static problem.

\begin{figure}[!t]
\begin{subequations} \label{model:sub}
    \begin{align}
\min	&  \quad \sum_{i \in P_v} (u_i - e_i) - \sum_{i \in P_v}\sum_{j \in \mathcal{N}_v} x_{ij} \pi_{i} - \sigma_v \label{model:sub_obj}\\ 
\intertext{subject to} 
    & \sum_{j \in \mathcal{N}_v} x_{ij} = \sum_{j \in \mathcal{N}_v} x_{ji} & \forall i \in \mathcal{N}_v\setminus\{0,s\} \label{model:sub_constr:flow} \\
    & \sum_{j \in \mathcal{N}_v} x_{0j} = 1  \label{model:sub_constr:source} \\
	& \sum_{j \in \mathcal{N}_v} x_{js} = 1  \label{model:sub_constr:sink} \\
	& \sum_{j \in \mathcal{N}_v} x_{ij} - \sum_{j \in \mathcal{N}_v}x_{n+i, j} = 0 & \forall i \in P_v \label{model:sub_constr:droppick} \\
	& \sum_{i \in \mathcal{N}_v} x_{ij} = 1 & \forall j \in I_v \label{model:sub_constr:dropoffs} \\
    & u_j \geq (u_i + \Delta_i + t_{ij})x_{ij} & \forall i,j \in \mathcal{N}_v \label{model:sub_constr:vstart} \\
    & u_0 \geq T^B_v \label{model:sub_constr:voperation1} \\
    & u_s \leq T^E_v \label{model:sub_constr:voperation2} \\
    & u_i \geq e_i & \forall i \in P_v \label{model:sub_constr:bounds} \\
    & t_i \leq u_{n+i} - (u_i + \Delta_i) \leq \max \{ \alpha t_i, \beta + t_i \}& \forall i \in P_v \label{model:sub_constr:traveltime} \\
    & t_{i} \leq u_i - (u^P_i + \Delta_i) \leq \max\{ \alpha t_{i}, \beta + t_{i}\} & \forall i \in I_v \label{model:sub_constr:traveltime_passengers} \\
    & w_j \geq (w_i + q_j) x_{ij} & \forall i,j \in \mathcal{N}_v \label{model:sub_constr:calccapacity} \\
    & 0 \leq w_i \leq Q_v & \forall i \in \mathcal{N}_v \label{model:sub_constr:capcity} \\
    & x_{ij} \in \{0, 1\} & \forall i,j \in \mathcal{N}_v \label{model:sub_constr:domain}
\end{align}
\end{subequations}
\caption{The Pricing Problem Formulation for Vehicle~$v$.} \label{fig:pricing}
\end{figure}

\paragraph{The Column Generation}

\newcommand{\tpmod}[1]{{\@displayfalse\pmod{#1}}}
\let\oldnl\nl
\newcommand{\nonl}{\renewcommand{\nl}{\let\nl\oldnl}}
\def\funname{\textsc{Release}}

\begin{algorithm}[!t]
    \caption{\textsc{ColumnGeneration}}
    \label{alg:cg}
    \setcounter{AlgoLine}{0}

    \DontPrintSemicolon
    \While{true} {
      $\mathcal{C} \gets ${\sc GenerateColumns}()\;
      \If{$\mathcal{C} = \emptyset$} {
        break; \;
      }
      Solve RMP after adding $\mathcal{C}$ \;
    }
    \SetKwFunction{FMain}{{\sc GenerateColumns}}
    \SetKwFunction{FSUB}{{\sc GenerateSizedColumms}}
    \SetKwProg{Fn}{Function}{:}{}
    \label{a1:fundecl}
    \nonl \Fn{\FMain{}}
          {
             $k \gets 1$ \;
             \While{$k \leq |P|$} {
                $\mathcal{C} \gets $ {\sc GenerateSizeColumns}($k$)\;
                \If{$\mathcal{C} \neq \emptyset$} {
                   \Return $\mathcal{C}$ \;
                }
                \Else {
                  $k$++ \;
                }
             }
          }
    \nonl \Fn{\FSUB{k}}
          {
             $Q \gets \{ R \subseteq P \mid |R| = k \}$ \;
             \Forall{$v \in |V|$ ordered by decreasing $\sigma_v$} {
               $R_v \gets \argmin_{R \subset Q}$ {\sc pricing}($v,R$)\;
                \If{$\mbox{{\sc pricing}}(v,R_v) \} < 0$} {
                   $Q \gets \{ R \subseteq Q \mid R \cap R_v = \emptyset \}$ \;
                }
             }
             \Return $\{ \mbox{\sc route}(v,R_v) | v \in V ~\&~ \mbox{{\sc pricing}}(v,R_v) < 0 \} $ \;
          }          
\end{algorithm}

In traditional column generation for dial-a-ride problems, the pricing
problem is formulated as a resource-constrained shortest-path problem
and solved using dynamic programming. However, the minimization of
waiting times, i.e., $\sum_{i \in P} (u_i - e_i)$, is particularly
challenging, as it cannot be formulated as a classical
resource-constrained shortest-path problem. One option is to
discretize time and use time-expanded graphs. However, this raises
significant computational challenges for large instances. As a result,
this paper solves the pricing problem through an anytime algorithm
that takes into account the real-time constraints \name{} operates
under.

The column-generation algorithm is specified in Algorithm
\ref{alg:cg}: {\em It generates multiple columns with disjoint sets of
  customers}. In the algorithm, function {\sc Pricing}($v,R$) solves
the pricing problem for a vehicle $v$ and a set $R$ of requests, while
{\sc Route}($v,R$) returns the optimal route for a vehicle $v$ and a
set of request $R$. Lines 1--5 is the high-level column-generation
procedure: It alternates the generation of columns and the solving of
the master problem with the generated columns until no more columns
can be generated. It proceeds in waves, first generating columns with
one customer before progressively increasing the number of considered
requests.  Procedure {\sc GenerateColumn} (lines 6--12) generates
columns by increasing number of requests. Procedure {\sc
GenerateSizedColumn} (lines 13--18) generates columns of size $k$,
where $k$ is the number of requests in the column.  It first computes
$Q$, a set in which each element is a $k$-sized set of possible
requests.  It then considers the various vehicles ranked in decreasing
order of their dual values $\sigma_v$. Line~15 computes the sets of
requests with the smallest pricing objective value. If the pricing
objective is negative (line~16), all set of requests which contains a
request covered by $R_v$ are removed from $Q$ to ensure that \name{}
generates a set of non-overlapping columns at each iteration
(line~17). Finally, line 18 returns the routes for each vehicle with
negative reduced costs.
\section{The Real-Time Problem}
\label{section-rt}

\name{} divides the time horizon into epochs of length $\ell$, i.e.,
$[0,\ell),[\ell,2\ell),[2\ell,3\ell),\ldots$ and epoct $\tau$
corresponds to the time interval $[\tau \ell , (\tau+1) \ell)$.
During period $\tau$, \name{} batches the incoming requests into a set
$P_\tau$, which is considered in the next epoch. It also optimizes the
static problem using the requests accummulated in $P_{\tau-1}$ and
those requests not yet committed to in the epochs $\tau-1$ and
before. The optimization is performed over the interval
$[(\tau+1)\ell,\infty)$.

It remains to specify how to compute the inputs to the optimization
problem, i.e., the departing stops and times for each vehicle and the
various set of requests to serve. To determine the starting stop for a
vehicle $v$, the optimization in epoch $\tau$ uses the solution
$\phi_{\tau-1}$ to the static problem in epoch $\tau-1$ and considers
the first stop $s_v$ in $\phi_{\tau-1}$ in the interval
$[(\tau+1)\ell,\infty)$ if it exists. This stop becomes the starting
  stop $u_0^v$ of the vehicle and its earliest time is given by the
  earliest departure time of vehicle $v$ in $\phi_{\tau-1}$. If
  vehicle $v$ is idle at stop $s_v$ in $\phi_{\tau-1}$ and not
  scheduled on $[(\tau+1)\ell,\infty)$, then the departing stop is
    $s_v$ and the earlierst departing time is $(\tau+1) \ell$.
    Consider now the sets $P$, $D$, and $I_v$ ($v \in V$) for period
    $\tau$.  For a vehicle $v$, all the requests before its departing
    stop $s_v$ are said to be {\em committed} and are not
    reconsidered.  The set $I_v$ are the dropoffs of the requests that
    have been picked up before $s_v$ but not yet dropped off. The set
    $P$ corresponds to the requests that have not been picked up by
    any vehicle $v$ before $s_v$, as well as the requests batched in
    $P_{\tau-1}$. The set $D$ simply contains the dropoffs associated
    with $P$.

Finally, since the static problem may not schedule all the requests,
it is important to update the penalty of unserved requests to ensure
that they will not be delayed too long. The penalty for an unserved
request $c$ in period $\tau$ is given by $ p_c = \delta 2^{(\tau \ell
  - e_c) / (10\ell)} \label{equation:obj} $ and it increases
exponentially over time as shown in Figure \ref{fig:obj_func}. The
$\delta$ parameter incentivizes the schedule of the request in its
first available period. Figure \ref{fig:obj_func} displays the
function for $\delta = 420$ seconds and $\ell = 30$ seconds: It
ensures that the penalty doubles every ten periods (in the example,
every five minutes).

\begin{figure}[!t]
    \centering
    \includegraphics[scale=0.4]{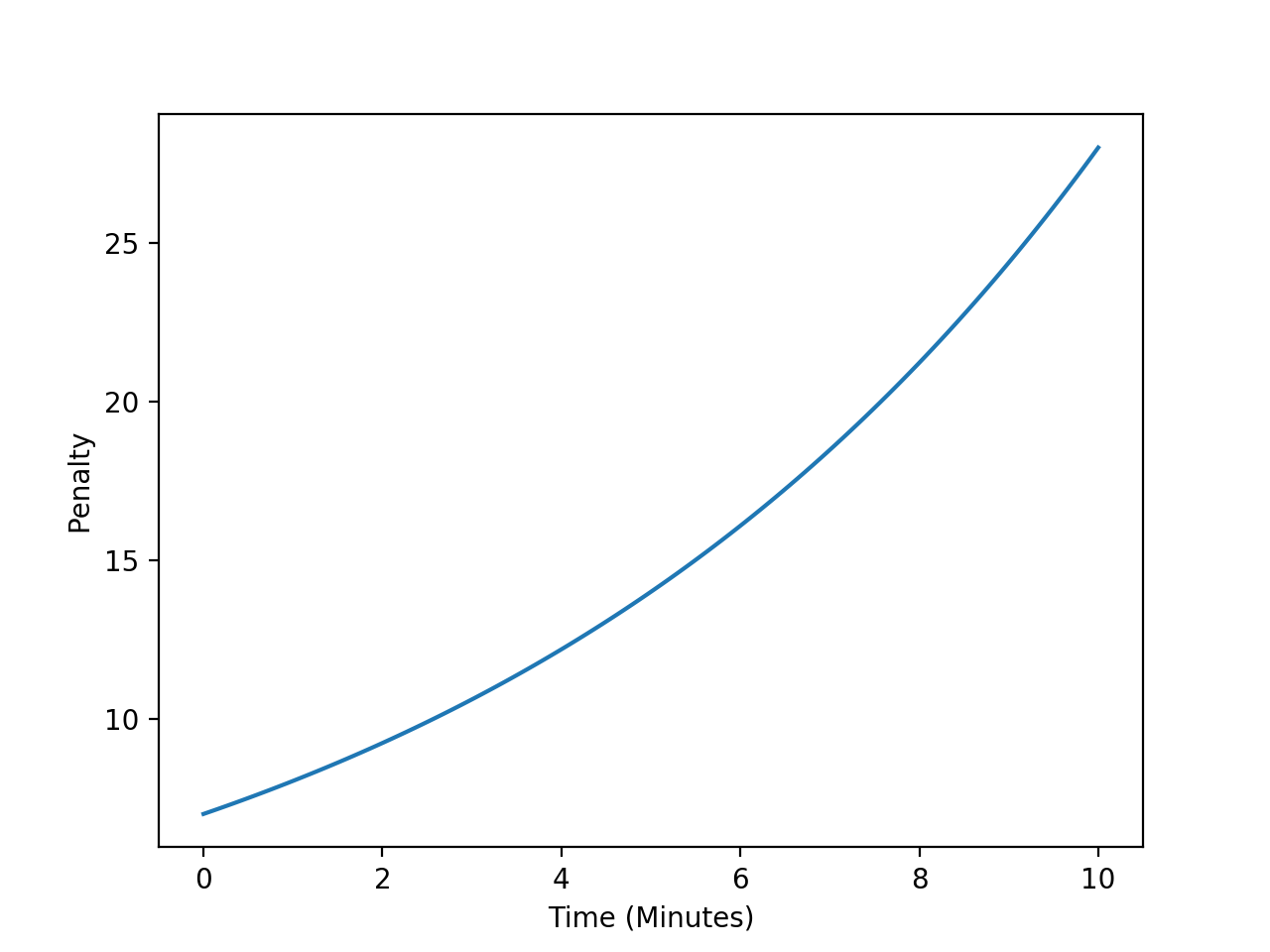}
    \caption{The Penalty Function for Unserved Customers.}
    \label{fig:obj_func}
\end{figure}

Observe that the static model schedules all the requests which have
not been committed to any vehicle. This gives a lot of flexibility to
the real-time system at the cost of more complex pricing subproblems.

\section{Experimental Results}
\label{section-results}
\paragraph{Instance Description}

\name{} was evaluated on the yellow trip data provided by the New York
City Taxi and Limousine Commission \cite{nycdata}.  This data provides
{\em pickup and dropoff locations}, which were used to match trips to
the closest virtual stops, {\em starting times}, which were used as
the request time, and the {\em number of passengers}.  This section
reports results on a representative set of 24 instances, 1 hour per
day for two weekdays per month from July 2015 through June 2016.  To
capture the true difficulty of the problem, rush hours (7--8am) were
selected.  The instances have an average of 21,326 customers and range
from 6,678 customers to 28,484 customers.  Individual requests with
more customers than the capacity of the vehicles were split into
several trips. An additional test was performed on the largest
instance with 32,869 customers.

\paragraph{Virtual Stops}

The evaluation assumes a dial-a-ride system using the concept of
virtual stops proposed in the {\sc RITMO} system \cite{RITMO} (Uber
and Lyft are now considering similar concepts). Virtual stops are
locations where vehicles can pick up and drop off customers without
impeding traffic. They also ensure that customers are ready to pick up
and make ride-sharing more efficient since they decrease the number of
stops. To implement virtual stops, Manhattan was overlayed with a grid
with cells of 200 squared meters and every cell had a virtual
stop. The trip times were precomputed by querying OpenStreetMap for
travel times between each virtual stop \cite{OpenStreetMap}.  All
customers at a virtual stop are grouped and can be picked up together.

\paragraph{Algorithmic Setting}

Both the final master problem and the restricted master problem are
solved using Gurobi 8.1. Empty vehicles are initially evenly
distributed over the virtual stops.  The pricing problem uses parallel
computing to implement line 15 of Algorithm \ref{alg:cg}, exploring
potential requests simultaneously. To meet real-time constraints, the
implementation greedily extends the ``optimal'' routes of size $k$ to
obtain routes of size $k+1$. Unless otherwise specified, all
experiments are performed with the following default parameters: 2,000
vehicles of capacity 5, $\alpha = 1.5$, $\beta = 240$ seconds, and
$\delta = 420$ seconds. The impact of these parameters is also
studied.

\paragraph{Wait Times}

\begin{figure}[!t]
    \centering
    \begin{minipage}{.45\textwidth}
        \centering
        \includegraphics[width=.9\linewidth]{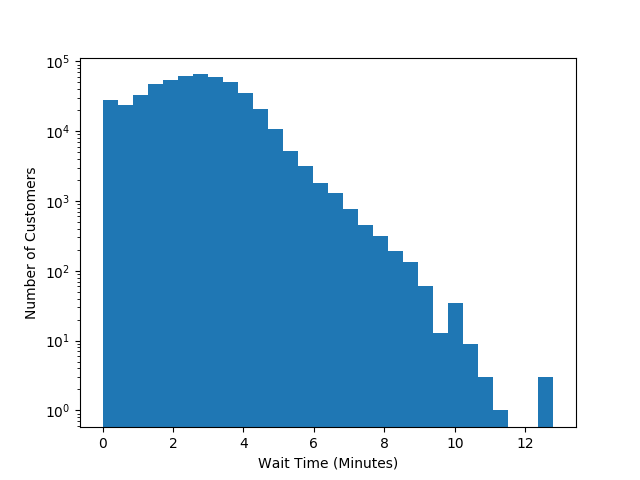}
        \captionof{figure}{The Histogram of Wait Times (Log Scale).}
        \label{fig:wait_hist}
    \end{minipage}%
    \hfill
    \begin{minipage}{.45\textwidth}
        \centering
        \includegraphics[width=.9\linewidth]{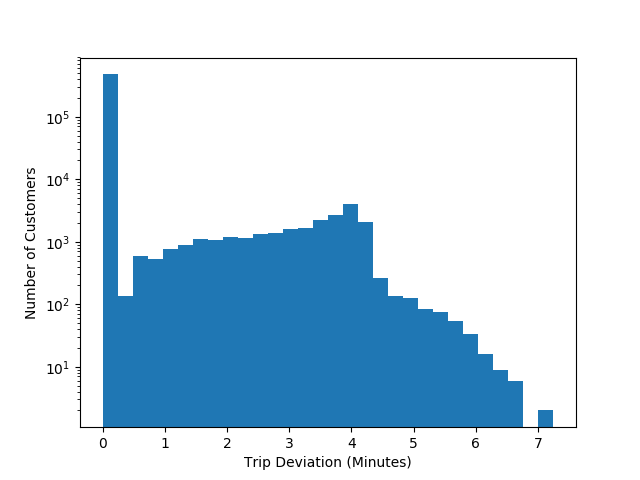}
        \captionof{figure}{The Histogram of Trip Deviations (Log Scale).}
        \label{fig:trip_hist}
    \end{minipage}
\end{figure}

Figure~\ref{fig:wait_hist} reports the distribution of the waiting for
all customers across all instances. The results demonstrate the
performance of \name{}: The average waiting time is
about 2.58 minutes with a standard deviation of 1.31.  On the instance
with 32,869 customers, the average waiting time is 5.42 minutes.

\paragraph{Trip Deviation}

Figure~\ref{fig:trip_hist} depicts a histogram of trip deviations
incurred because of ride-sharing. The results indicate that riders
have an average trip deviation of 0.34 minutes with a standard
deviation of 0.74. In percentage, this represents a deviation of about
12\%. On the instance with 32,869 customers, the average trip deviation
is 2.23 minutes, which shows the small overhead induced by
ride-sharing.

\paragraph{The Impact of the Fleet Size}

Figure~\ref{fig:fleetsize} studies the impact of the fleet size on the
waiting times and trip deviation. The plot reports the average waiting
times for various numbers of riders, where capacity is 4, $\alpha =
1$, $\beta = 840$ seconds, and $\delta = 420$ seconds to facilitate 
comparisons to \cite{Alonso-Mora462}.  The results
show that, even with 1,500 vehicles, the average waiting time remains
below 6 minutes and the average deviation time below 40 seconds. Since
\name{} is guaranteed to serve all the requests, these results
demonstrate the potential of column generation and ride-sharing for
large-scale real-time dial-a-ride platforms. Adopting \name{} has the
potential to substantially reduce traffic in large cities, while still
guaranteeing service within reasonable times. Recall that the approach
in \cite{Alonso-Mora462} does not serve about 2\% of the requests.

\begin{figure}[!t]
    \centering
    \subfloat[The Impact on the Average Wait Times.]{{\includegraphics[width=5cm]{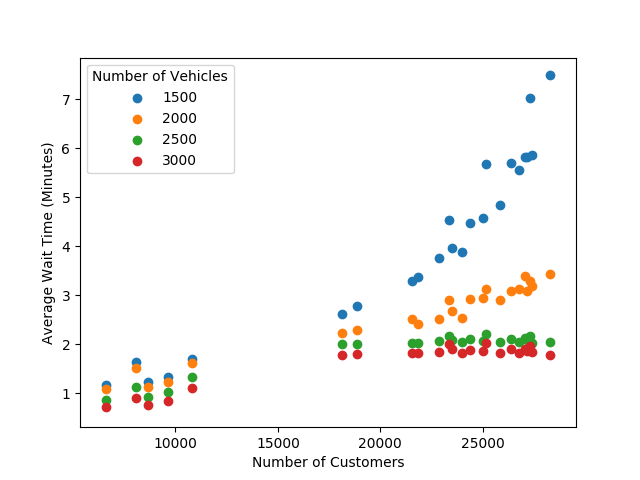} }}
    \qquad
    \subfloat[The Impact on the Average Trip Deviations.]{{\includegraphics[width=5cm]{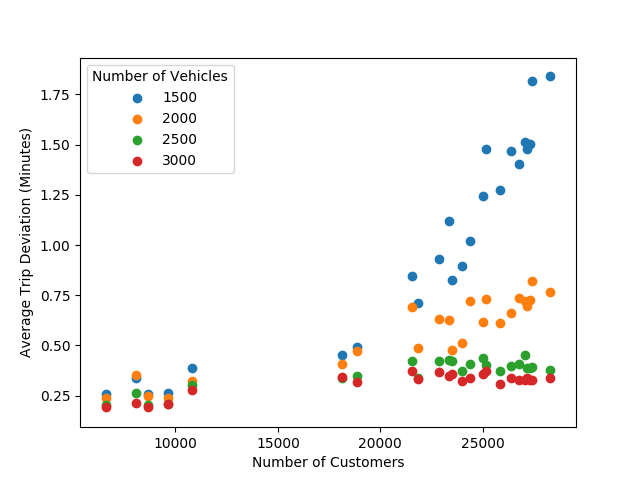} }}
    \caption{The Impact of the Fleet Size on the Average Wait Times and Average Deviations on All Instances.}
    \label{fig:fleetsize}%
\end{figure}

\paragraph{The Impact of Vehicle Capacity}

Figure~\ref{fig:capacity} studies the impact of the vehicle capacity
(i.e., how many passengers a vehicle can carry) on the average waiting
times and trip deviation. The parameters are set to 2,000 vehicles, 
$\alpha = 1$, $\beta = 840$ seconds, and $\delta = 420$ seconds
to facilitate comparisons to \cite{Alonso-Mora462}. 
The results on waiting times show that moving to
vehicles of capacity 8 further reduces the average waiting times,
especially on the large instances. On the other hand, moving from a
capacity 5 to 3 does not affect the results too much. The results on
deviations are more difficult to interpret. Obviously moving to a
capacity 8 further increases the deviation (although it remains below
one minute). However, moving to vehicles of capacity 3 also increases
the deviation, which is not intuitive. This may be a consequence of 
myopic decisions that cannot be corrected easily given the tight
capacity. 

\begin{figure}[!t]
    \centering
    \subfloat[The Impact on Wait Times.]{{\includegraphics[width=5cm]{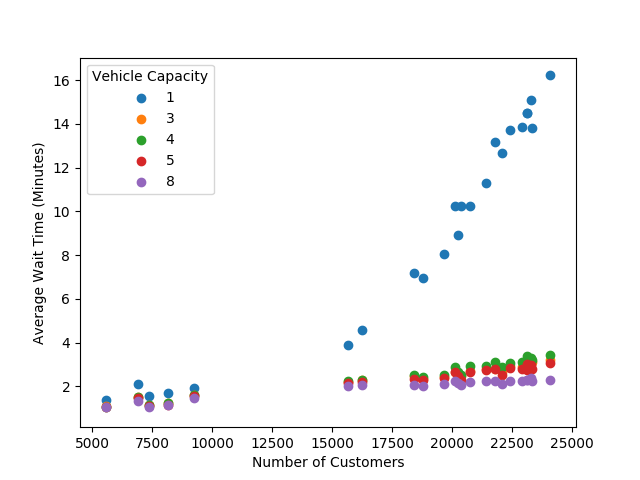}}}
    \qquad
    \subfloat[The Impact on Trip Deviations.]{{\includegraphics[width=5cm]{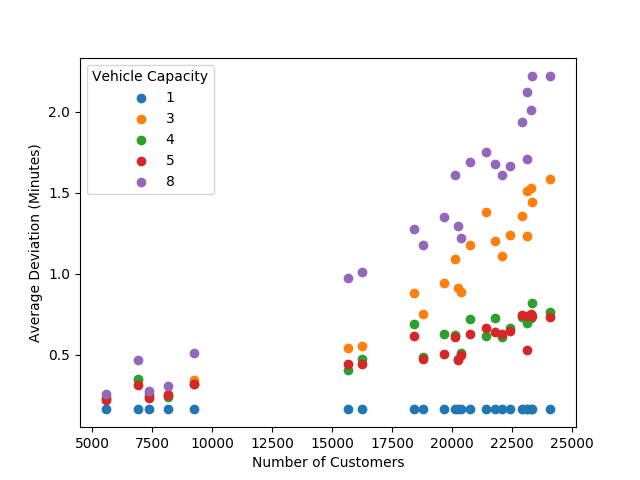}}}
    \caption{The Impact of the Vehicle Capacity on the Average Wait Times and the Average Trip Deviations on All Instances.}%
    \label{fig:capacity}%
\end{figure}

\paragraph{The Impact of the Penalty}

The penalty $p_i$ in the model is an exponential function of the
current waiting time of customer $i$. Constant $\delta$ controls the
initial penalty: If it is too small, the penalty for not scheduling a
request for the first few periods is low, which causes an increase in
wait times, as can be observed in Figure~\ref{fig:obj_vs_wait}. Once
$\delta$ is large enough, the average wait times converge to the same
values.

\begin{figure}[!t]
    \centering
    \begin{minipage}{.45\textwidth}
        \centering
        \includegraphics[width=.9\linewidth]{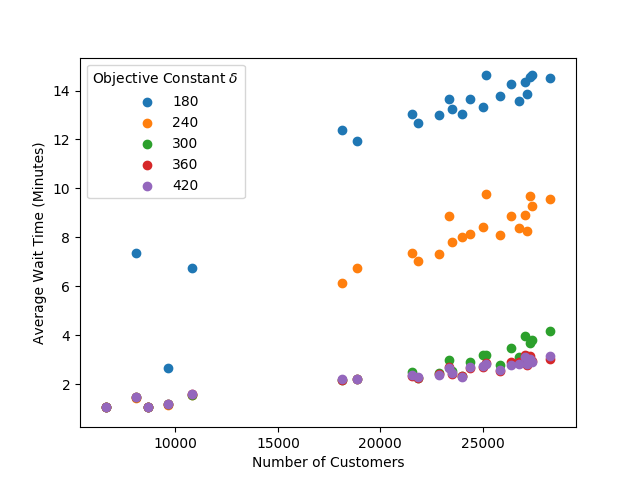}
        \captionof{figure}{The Impact of the Penalty on Average Wait Times.}
        \label{fig:obj_vs_wait}
    \end{minipage}%
    \hfill
    \begin{minipage}{.45\textwidth}
        \centering
        \includegraphics[width=.9\linewidth]{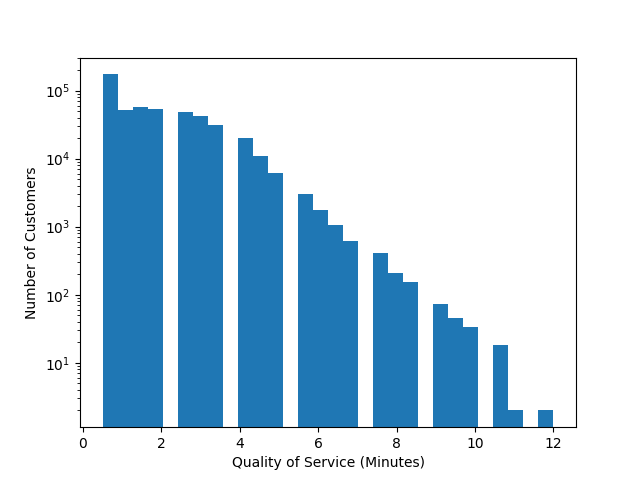}
        \captionof{figure}{Times Until Final Vehicle Assignments.}
        \label{fig:qos_hist}
    \end{minipage}%
\end{figure}

\paragraph{Final Vehicle Assignments}

As a result of re-optimization, the vehicle to which a rider is
assigned can change.  Figure~\ref{fig:qos_hist} reports the amount of
time until riders receive their final vehicle assignment (the vehicle
which actually picks them up).  Not surprisingly, this histogram
closely follows the waiting time distribution.  The majority of riders
receive this assignment quickly.  However, it takes some riders over
10 minutes to receive their final vehicle assignment, which shows that
\name{} takes advantage of the ability to re-assign riders
to vehicles which will result in better overall assignments.

\begin{figure}[!t]
  \centering
    \begin{minipage}{.45\textwidth}
        \centering
        \includegraphics[width=.9\linewidth]{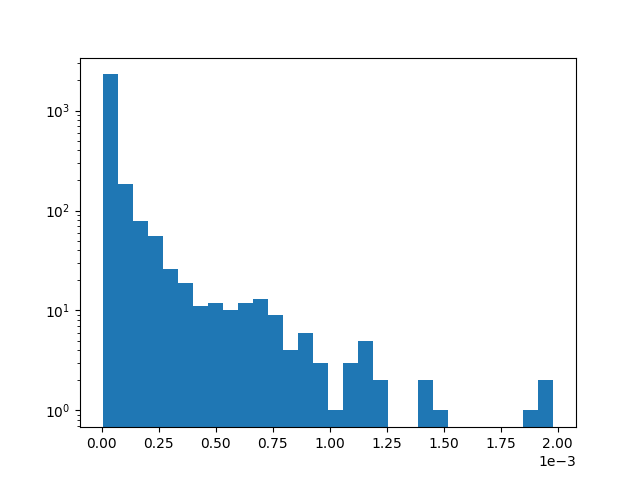}
        \captionof{figure}{The Number of Generated Columns as a Percentage of Possible Combinations of Requests/Vehicles. The x-Axis value are scaled by $10^{-3}$.}
        \label{fig:cg_hist}
    \end{minipage}
  \hfill
        \begin{minipage}{.45\textwidth}
        \centering
        \includegraphics[width=.9\linewidth]{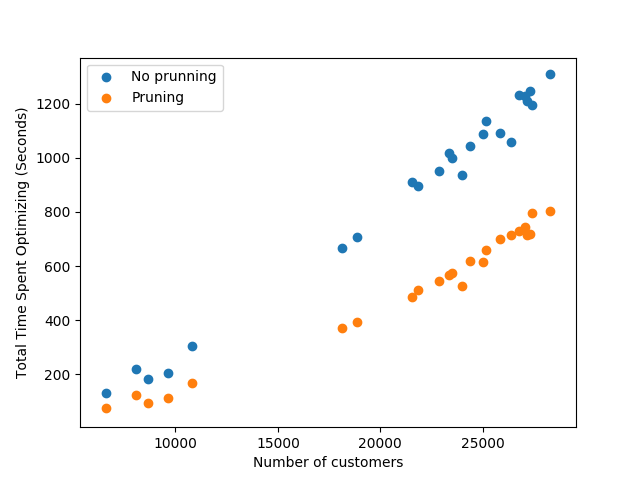}
        \captionof{figure}{Optimization Times With and Without Pruning.}
        \label{fig:pruning_plot}
    \end{minipage}
\end{figure}

\paragraph{The Impact of Column Generation}

Figure \ref{fig:cg_hist} depicts the impact of column generation and
reports the number of columns in the final MIP as all possible columns
of sizes 1 and 2 to be conservative. The results show that the
algorithm only explores a small percentage of all potential columns,
demonstrating the benefits of a column-generation approach.

\paragraph{The Impact of Pruning}

Figure~\ref{fig:pruning_plot} shows the impact of Theorem 1, which
provides a way to prune the number of requests considered at each step
of the algorithm. The figures report the total optimization time for
all time periods of each instance. Each optimization must be performed
in less than 30 seconds, but the graph reports the total optimization
time over the entire hour. As the results indicate, the pruning
benefits become substantial as the instance sizes grow. The results
show that the pruning significantly reduces the computational
time. They also show that \name{} should be able to handle even
larger instances since, after exploiting Theorem 1, \name{} uses
only about a sixth of the available time. This creates opportunities
to exploit stochastic information.

\begin{figure}[!t]
    \centering
    \subfloat[The Impact on Average Vehicle Utilization.]{{\includegraphics[width=5cm]{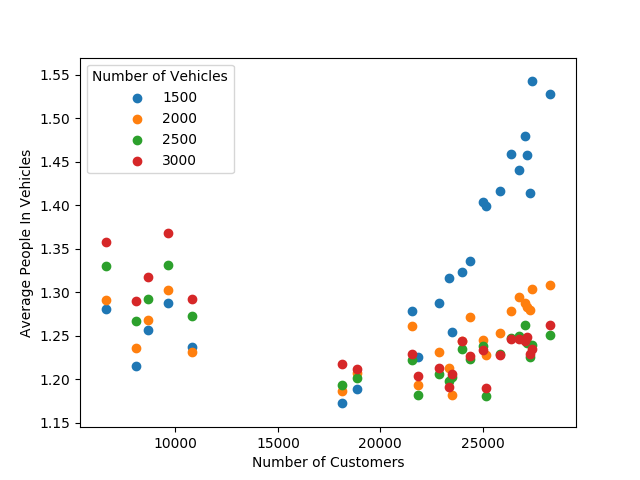} }}
    \qquad
    \subfloat[The Impact on Average Vehicle Idle Time.]{{\includegraphics[width=5cm]{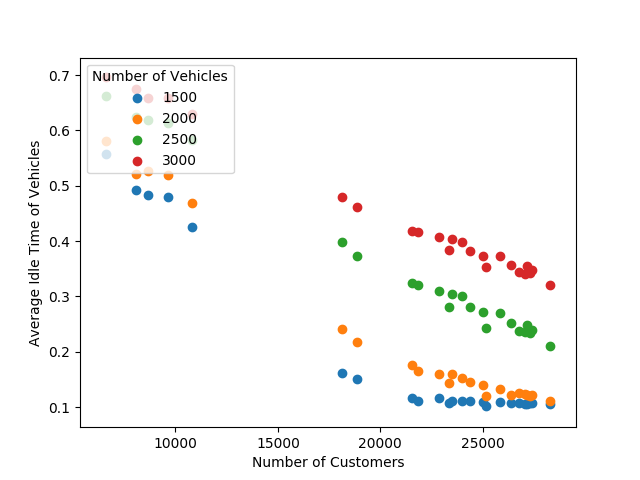} }}
    \caption{The Impact of the Fleet Size on the Average Vehicle Utilization and Idle Time on All Instances.}
    \label{fig:vehicle_util}%
\end{figure}

\paragraph{The Impact of Ride Sharing}

Figure~\ref{fig:vehicle_util} reports the average number of people in
each vehicle at all times for each instance.  The results show a
significant amount of ride sharing, although single trips and idle
time remain a significant portion of the rides, especially when the
fleet is oversized.  Lastly, Figure~\ref{fig:capacity} shows that wait
times are reduced by a factor of 4 when moving from single-rider trips
to ride-sharing for large instances while the trip deviation only
increases to at most 2 minutes for vehicles of capacity 8, thus
demonstrating the value of ride sharing.

\paragraph{Comparison with Prior Work}

The results of \cite{Alonso-Mora462} ``show that 2,000 vehicles (15\%
of the taxi fleet) of capacity 10 or 3,000 of capacity 4 can serve
98\% of the demand within a mean waiting time of 2.8 min and mean trip
delay of 3.5 min.''  \name{} relaxes the hard time-windows present in
\cite{Alonso-Mora462} and improves on these results, yielding an
average wait time of 2.2 minutes with only 2,000 vehicles, while
guaranteeing service for all riders.

\section{Conclusion}
\label{section-conclusion}

This paper considered the real-time dispatching of large-scale
ride-sharing services over a rolling horizon. It presented \name{}, a
real-time optimization framework that divides the time horizon into
epochs and uses a column-generation algorithm that minimizes wait
times while guaranteeing services for every rider and a small trip
deviation compared to a direct trip. This contrasts to earlier work
which rejected customers when the predicted waiting time was
considered too long (e.g., 7 minutes). This assumption reduced the
search space at the cost of rejecting a significant number of
requests.

The column-generation algorithm of \name{} is derived from a
three-index formulation \cite{Cordeau2007} which is adapted for use in
real-time dial-a-ride applications. In addition, to ensure that all
riders are served in reasonable times, the paper proposed an
optimization model that balances the minimization of waiting times
with penalties for riders that are not scheduled yet. These penalties
are increased after each epoch to make it increasingly harder not to
serve waiting riders. The paper also presented a key property of the
formulation that makes it possible to reduce the search space
significantly.

\name{} was evaluated on historic taxi trips from the New York City
Taxi and Limousine Commission \cite{nycdata}, which contains
large-scale instances with more than 30,000 requests an hour. The
results indicated that \name{} enables a real-time dial-a-ride service
to provide service guarantees (every rider is served in reasonable
time) while improving average waiting times and average trip
deviations compared to prior work.  The results also showed that
larger occupancy vehicles bring benefits and that the fleet size can
be further reduced while preserving very reasonable waiting times.

Substantial work remains to be done to understand the strengths and
limitations of the approach. The current implementation is myopic and
heavily driven by the dual costs to generate the columns. Different
pricing implementation, including the use of constraint programming to
replace our dedicated search algorithm, and the inclusion of
stochastic information are natural directions for future research.

\section*{Acknowledgments}

This research was partly supported by Didi Chuxing Technology Co. and
Department of Energy Research Grant 7F-30154. We would like to thank
the reviewers for their detailed comments and suggestions which
dramatically improved the paper, and the program chairs for a rebuttal
period that was long enough to run many experiments.

\bibliographystyle{acm}
\bibliography{references}

\begin{thebibliography}{10}

\bibitem{Alonso-Mora462}
{\sc Alonso-Mora, J., Samaranayake, S., Wallar, A., Frazzoli, E., and Rus, D.}
\newblock On-demand high-capacity ride-sharing via dynamic trip-vehicle
  assignment.
\newblock {\em Proceedings of the National Academy of Sciences 114}, 3 (2017),
  462--467.

\bibitem{Bent2007}
{\sc Bent, R., and Van~Hentenryck, P.}
\newblock Waiting and relocation strategies in online stochastic vehicle
  routing.
\newblock In {\em Proceedings of the 20th International Joint Conference on
  Artifical Intelligence\/} (San Francisco, CA, USA, 2007), IJCAI'07, Morgan
  Kaufmann Publishers Inc., pp.~1816--1821.

\bibitem{scenariopvh}
{\sc Bent, R.~W., and Van~Hentenryck, P.}
\newblock Scenario-based planning for partially dynamic vehicle routing with
  stochastic customers.
\newblock {\em Operations Research 52}, 6 (2004), 977--987.

\bibitem{Berbeglia2012}
{\sc Berbeglia, G., Cordeau, J.-F., and Laporte, G.}
\newblock A hybrid tabu search and constraint programming algorithm for the
  dynamic dial-a-ride problem.
\newblock {\em INFORMS Journal on Computing 24}, 3 (2012), 343--355.

\bibitem{Bertsimas2018OnlineVR}
{\sc Bertsimas, D., Jaillet, P., and Martin, S.}
\newblock Online vehicle routing : The edge of optimization in large-scale
  applications.

\bibitem{Cordeau2007}
{\sc Cordeau, J.-F., and Laporte, G.}
\newblock The dial-a-ride problem: models and algorithms.
\newblock {\em Annals of Operations Research 153}, 1 (Sep 2007), 29--46.

\bibitem{Jain2011}
{\sc Jain, S., and Van~Hentenryck, P.}
\newblock Large neighborhood search for dial-a-ride problems.
\newblock In {\em International Conference on Principles and Practice of
  Constraint Programming\/} (2011), Springer, pp.~400--413.

\bibitem{nycdata}
{\sc NYC}.
\newblock Nyc taxi \& limousine commission - trip record data.

\bibitem{OpenStreetMap}
{\sc {OpenStreetMap contributors}}.
\newblock Planet dump retrieved from https://planet.osm.org, 2017.

\bibitem{scalable-taxi}
{\sc Ota, M., Vo, H., Silva, C., and Freire, J.}
\newblock A scalable approach for data-driven taxi ride-sharing simulation.
\newblock In {\em 2015 IEEE International Conference on Big Data (Big Data)\/}
  (Oct 2015), pp.~888--897.

\bibitem{stars}
{\sc Ota, M., Vo, H., Silva, C., and Freire, J.}
\newblock Stars: Simulating taxi ride sharing at scale.
\newblock {\em IEEE Transactions on Big Data 3}, 3 (Sept 2017), 349--361.

\bibitem{RITMO}
{\sc www.secondwavemedia.com/concentrate/innovationnews/ritmorollout0443.aspx}.
\newblock {RITMO app introduces on-demand mass transit at U-M, with plans to
  expand}.
\newblock Concentrate, 2018.

\end{thebibliography}
\end{document}